\documentclass{svjour3}                     
\smartqed  
\usepackage{latexsym}
\usepackage{amssymb}
\usepackage{amsmath}
\usepackage{rotating}
\usepackage[misc]{ifsym}

\newtheorem{them}{Theorem}[section]
\newtheorem{lem}{Lemma}[section]

\newtheorem{coro}{Corollary}[section]
\newtheorem{pro}{Proposition}[section]

\journalname{Set-Value Var. Anal}
\begin{document}

\title{BCQ and Strong BCQ for Nonconvex Generalized Equations with Applications to Metric Subregularity}

\titlerunning{BCQ and Strong BCQ for Nonconvex Generalized Equations with Applications}



\author{Liyun Huang\and Qinghai He\and Zhou Wei}


\institute{Liyun Huang \at School of Mathematics and Information Science, Qujing Normal University,\\ Qujing 655011,Yunnan Province, P. R. China\\ \email{ynszlyb@163.com}
  \and Qinghai He \at Department of Mathematics, Yunnan University,
 Kunming 650091, P. R. China\\ \email{heqh@ynu.edu.cn}\and
 Zhou Wei(\Letter) \at Department of Mathematics, Yunnan University,
 Kunming 650091, P. R. China\\ \email{wzhou@ynu.edu.cn}}

\date{Received: date / Accepted: date}

\maketitle

\begin{abstract}
In this paper, based on basic constraint qualification (BCQ) and strong BCQ for convex generalized equation, we are inspired to further discuss constraint qualifications of BCQ and strong BCQ for nonconvex generalized equation and then establish their various characterizations. As applications, we use these constraint qualifications to study metric subregularity of nonconvex generalized equation and provide necessary and/or sufficient conditions in terms of constraint qualifications considered herein to ensure nonconvex generalized equation having metric subregularity.


\end{abstract}

\section{Introduction}
Let $X$ and $Y$ be Banach spaces and $F:X \rightrightarrows Y$ be a
closed multifunction, and let $A$ be a closed subset of $X$ and $b$ be a given point in $Y$.
Consider the following generalized equation with constraint (GEC)
$$
b\in F(x)\,\,\mathrm{subject}\,\,\mathrm{to}\,\,x\in A. \eqno{\rm (GEC)}
$$
This paper is devoted to several concepts of constraint qualifications for (GEC) and applications to metric subregularity of (GEC).

It is well known that basic constraint qualification (BCQ) for
continuous convex inequalities is a fundamental concept in mathematical programming, and has been extensively studied by many authors. Readers could consult references \cite{BB,Li,HL,LN,LNS,W,WZ,11,9,ZW} for the details on BCQ as well as its close relationship with other important concepts in optimization. In 2004, Zheng and Ng \cite{11} made use of singular subdifferential to introduce the concept of strong BCQ which is strictly stronger than BCQ for the convex inequality defined by one lower semicontinuous convex function, and used this notion to characterize metric regularity of the convex inequality. Afterwards, Hu \cite{10} further studied strong BCQ and introduced one measurement of end set to provide equivalent conditions for strong BCQ. In 2007, Zheng and Ng \cite{7} generalized the concepts of BCQ and strong BCQ to the case of convex generalized equation and used these constraint qualifications to obtain necessary and sufficient conditions for convex generalized equation to have metric subregularity. Naturally, it is interesting and important to further consider constraint qualifications as well as applications for nonconvex generalization equation. Motivated by this and as one aim of this paper, we are inspired by \cite{11} and \cite{10} to discuss BCQ and strong BCQ for (GEC) as well as their characterizations and apply these constraint qualifications to the study on metric subregularity of (GEC).

Metric subregularity is a well-known and useful concept in mathematical programming and optimization, and has been extensively studied by many authors under various names (cf. \cite{9,4,BS,5,13,29,22,23,28} and references therein). Note that Zheng and Ng \cite{7} discussed metric subregularity for convex generalized equation and provided dual characterizations for metric subregularity in terms of coderivative and normal cone. Recently the authors \cite{HYZ} considered metric subregularity for subsmooth generalized constraint equation. Based on \cite{7,HYZ} and as the other aim of this paper, we further investigate metric subregularity of (GEC) and mainly establish several necessary and/or sufficient conditions for (GEC) to have metric subregularity. These conditions are given in terms of constraint qualifications studied in this paper.

Given a Banach space $X$ with the dual space $X^*$ and a multifunction $\Phi: X\rightrightarrows X^*$, the symbol
\begin{equation*}
\begin{array}r
\mathop{\rm Limsup}\limits_{y\rightarrow x}\Phi(x):=\Big\{x^*\in X^*: \exists \ {\rm sequences} \ x_n\rightarrow x \ {\rm and} \ x_n^*\stackrel{w^*}\longrightarrow x^* \ {\rm with}\  \\
x_n^*\in \Phi(x_n) \ {\rm for\ all \ } n\in \mathbb{N} \Big\}
\end{array}
\end{equation*}
signifies the  sequential Painlev\'{e}-Kuratowski outer/upper limit of $\Phi(x)$ as $y\rightarrow x$.

\section{Preliminaries}
Let $X, Y$ be Banach spaces with the closed unit balls denoted by
$B_{X}$ and $B_Y$, and let $X^*, Y^*$ denote the dual space of $X$
and $Y$ respectively. For a closed subset $A$ of $X$, let $\overline
A$ denote the closure of $A$. For $a\in A$, let $T_c(A,a)$ and $T(A, a)$ denote the
Clarke tangent cone and the contingent (Bouligand) cone of $A$ at $a$
respectively, which are defined by
$$
T_c(A,a):=\liminf\limits_{x\stackrel{A}{\rightarrow}a,t\rightarrow
0^+}\frac{A-x}{t}\;\;\mathrm{and}\;\;T(A,a):=\limsup\limits_{t\rightarrow 0^+}\frac{A-a}{t},
$$
where $x\stackrel{A}{\rightarrow}a$ means that $x\rightarrow a$ with
$x\in A$. Thus, $v\in T_c(A,a)$ if and only if for any
$a_n\stackrel{A}{\rightarrow}a$ and any $t_n\rightarrow 0^+$, there exists $v_n\rightarrow v$ such that $a_n+t_nv_n\in A$ for all $n$, and
$v\in T(A,a)$ if and only if there exist $v_n\rightarrow v$ and $t_n\rightarrow 0^+$ such that $a+t_nv_n\in A$ for all $n$.

We denote by $N_c(A,a)$ the Clarke normal cone of $A$ at $a$, that is,
$$N_c(A,a):=\{x^*\in X^*:\;\langle x^*,h\rangle\leq0\;\;\forall h\in
T_c(A,a)\}.$$
Let $\hat N(A,a)$ denote the Fr\'{e}chet normal cone of $A$ at $a$ which is defined by
$$
\hat N(A,a):=\left\{x^*\in
X^*:\limsup\limits_{y\stackrel{A}\rightarrow a}\frac{\langle x^*,
y-a\rangle}{\|y-a\|}\leq 0\right\},
$$
and let $N(A, a)$ denote the Mordukhovich (limiting/basic) normal cone of $A$ at $a$ which is defined by
$$
N(A, a):=\mathop{\rm Limsup}_{x\stackrel A\rightarrow a, \varepsilon\downarrow
0}\hat N_{\varepsilon}(A, x),
$$
where $\hat N_{\varepsilon}(A, x)$ is the set of $\varepsilon$-normal to $A$ at
$x$ and defined as
$$
\hat N_{\varepsilon}(A, x):=\left\{x^*\in
X^*:\limsup\limits_{y\stackrel{A}\rightarrow x}\frac{\langle x^*,
y-x\rangle}{\|y-x\|}\leq \varepsilon\right\}.
$$
It is known from  \cite{17} and \cite{18} that
$$\hat N(A,a)\subset N(A, a)\subset N_c(A,a).$$
If $A$ is convex, all normal cones coincide and reduce to the normal cone in the sense of convex analysis; that is
$$
N_c(A,a)=N(A, a)=\hat{N}(A,a)=\{x^*\in X^*:\;\langle
x^*,x-a\rangle\leq 0\;\; \forall x\in A\}.
$$

For the case when $X$ is an Asplund space (cf. \cite{16} for definitions and their equivalences), Mordukhovich and
Shao \cite{18} have proved that
\begin{equation}\label{2.2a}
  N_c(A, a)=\overline{\rm co}^{w^*}(N(A,
a))\;\;\mathrm{and}\;\;N(A, a)=\mathop{\rm Limsup}_{x\stackrel A\rightarrow
a}\hat N(A, x)
\end{equation}
where $\overline{\rm co}^{w^*}$ denotes the weak$^*$ closed convex hull. This means $x^*\in N(A, a)$ if and only if there exist $x_n\stackrel{A}\rightarrow a$ and $x^*_n\stackrel{w^*}\rightarrow x^*$ such that $x_n^*\in \hat N(A, x_n)$ for all $n$.

Let $F:X\rightrightarrows Y$ be a multifunction. Recall that $F$ is said to be closed if ${\rm gph}(F)$ is a closed subset of $X\times Y$, where ${\rm gph}(F):=\{(x, y)\in X\times Y : y\in F(x)\}$ is the graph of $F$. Let $(x, y)\in {\rm gph}(F)$. Recall that the Clarke tangent derivative $D_cF(x, y)$ of $F$ at $(x, y)$ is defined by
$$
{\rm gph}(D_cF(x, y)):=T_c({\rm gph}(F),(x, y) ).
$$
Let $\hat D^*F(x, y), D^*F(x, y), D_c^*F(x, y): Y^*\rightrightarrows X^*$ denote Fr\'echet, Mordukhovich and Clarke coderivatives of $F$ at $(x, y)$ respectively, and they are defined as
$$
\begin{array}l
\hat D^*F(x, y)(y^*):=\{x^*\in X^* : (x^*, -y^*)\in \hat N({\rm gph}(F), (x, y))\},\\
D^*F(x, y)(y^*):=\{x^*\in X^* : (x^*, -y^*)\in N({\rm gph}(F), (x, y))\},\\
D_c^*F(x, y)(y^*):=\{x^*\in X^* : (x^*, -y^*)\in N_c({\rm gph}(F), (x, y))\}.
\end{array}
$$

Let  $\phi :X\rightarrow \mathbb{R}\cup\{+\infty\}$ be a proper
lower semicontinuous function and $x\in
\mathrm{dom}(\phi):=\{y\in X: \phi(y)<+\infty\}$. We denote Fr\'{e}chet, Mordukhovich and Clarke subdifferentials of $\phi$ at $x$ by $\hat{\partial}\phi(x), \partial\phi(x)$ and $ \partial_c\phi(x)$, respectively, which are defined as
\begin{eqnarray*}
\hat{\partial}\phi(x):=\{x^*\in X^* : (x^*, -1)\in \hat N({\rm epi}(\phi), (x, \phi(x)))\},\\
\partial\phi(x):=\{x^*\in X^* : (x^*, -1)\in N({\rm epi}(\phi), (x, \phi(x)))\},\\
\partial_c\phi(x):=\{x^*\in X^* : (x^*, -1)\in N_c({\rm epi}(\phi), (x, \phi(x)))\},
\end{eqnarray*}
where ${\rm epi}(\phi):=\{(x, \alpha)\in X\times \mathbb{R}: \phi(x)\leq \alpha\}$ denotes the epigraph of $\phi$. It is known that
$$
\hat{\partial}\phi(x)\subset\partial\phi(x)\subset\partial_c\phi(x).
$$
Further, one can verify that
$$
\hat{\partial}\phi(x)=\left\{x^*\in X^*:\;\liminf\limits_{z\rightarrow x}
\frac{\phi(z)-\phi(x)-\langle x^*,z-x\rangle}{\|z-x\|}\geq0\right\}.
$$
and
$$\partial_c \phi(x)=\{x^*\in X^*:\;\langle x^*,h\rangle
\leq \phi^{\circ}(x;h)\;\;\;\forall h\in X\},$$
here $\phi^{\circ}(x;h)$ denotes the generalized Rockafellar directional derivative of $\phi$
at $x$ along the direction $h$ and is defined by
$$
\phi^{\circ}(x;h):=\lim\limits_{\varepsilon\downarrow
0}\limsup\limits_{z\stackrel{\phi}{\rightarrow}x,
\,t\downarrow0}\inf_{w\in h+\varepsilon
B_{X}}\frac{\phi(z+tw)-\phi(z)}{t},
$$
where  $ z\stackrel{\phi}\rightarrow x$ means that $z\rightarrow x$
and $\phi(z)\rightarrow \phi(x)$. When $\phi$ is locally
Lipschitzian around $x$, $\phi^{\circ}(x;h)$ reduces to Clarke
directional derivative; that is
\begin{equation}\label{2.2}
\phi^{\circ}(x;h)=\limsup\limits_{z\rightarrow x,
\,t\downarrow0}\frac{\phi(z+th)-\phi(z)}{t}.
\end{equation}

For the case that $X$ is an Asplund space, Mordukhovich and
Shao \cite{18} have proved that
$$
\partial\phi(x)=\mathop{\rm Limsup}_{y\stackrel \phi\rightarrow
x}\hat{\partial}\phi(y);
$$
that is, $x^*\in\partial\phi(x)$ if and only if there exist $x_n\stackrel{\phi}\rightarrow x$ and $x_n^*\stackrel{w^*}\rightarrow x^*$ such that $x_n^*\in\hat{\partial}\phi(x_n)$ for all $n$.

The following lemmas will be used in our analysis. Readers are invited to consult references \cite{15} and \cite{14} respectively for more details.

\begin{lem}
Let $X$ be a Banach (resp. an Asplund) space and  $A$ be a nonempty closed
subset of $X$. Let $\gamma\in (0,\;1)$. Then for any $x\not\in A$
there exist $a\in A$ and $a^*\in N_c(A,a)$ (resp. $a^*\in
\hat{N}(A,a)$) with $\|a^*\|=1$ such that
$$\gamma\|x-a\|<\min\{d(x,A),\;\langle a^*,x-a\rangle\}.$$
\end{lem}

\begin{lem}
Let $X$ be an Asplund space and $A$ be a nonempty closed subset of
$X$. Let $x\in X\backslash A$ and $x^*\in \hat{\partial}d(\cdot,
A)(x)$. Then, for any $\varepsilon>0$ there exist $a\in A$ and
$a^*\in \hat N(A, a)$ such that
$$
\|x-a\|<d(x,
A)+\varepsilon\,\,\,\,\mathrm{and}\,\,\,\,\|x^*-a^*\|<\varepsilon
$$
\end{lem}

As one suitable substitute of convexity in this paper, we consider the concept of subsmooth which is introduced by Aussel, Daniilidis and Thibault \cite{2}. Recall that $A$ is said to be subsmooth at $a\in A$, if for any $\varepsilon>0$ there exists
$\delta>0$ such that
$$
\langle x^*-u^*, x-u\rangle\geq -\varepsilon\|x-u\|
$$
whenever $x, u\in B(a, \delta)\cap A$, $x^*\in N_c(A, x)\cap
B_{X^*}$ and $u^*\in N_c(A, u)\cap B_{X^*}$.

 When $A$ is subsmooth at $a\in A$, one has
$$
N_c(A, a)=N(A, a)=\hat N(A, a).
$$
Further, Zheng and Ng \cite{15} provided one characterization for
this notion; that is, $A$ is subsmooth at $a\in A$ if and only if for any $\varepsilon>0$ there exists
$\delta>0$ such that
\begin{equation}\label{2.2}
\langle u^*, x-u\rangle\leq d(x, A)+\varepsilon\|x-u\|\;\;\forall
x\in B(a, \delta)
\end{equation}
whenever $u\in B(a, \delta)\cap A$ and $u^*\in N_c(A, u)\cap
B_{X^*}$. Readers are invited to consult \cite[Proposition 2.1]{15} and \cite[Proposition 3.1]{ZW} for more details.

For a closed multifunction, Zheng and Ng \cite{9} introduced the concept of L-subsmooth and studied calmness for this kind of multifunctions. Recall from \cite{9} that a closed multifunction $F:X\rightrightarrows Y$ is said
to be

(i) L-subsmooth at $(a, b)\in {\rm gph}(F)$ if for any $\varepsilon>0$ there exists
$\delta>0$ such that
\begin{equation}
\langle u^*, x-a\rangle+\langle v^*,
y-b\rangle\leq\varepsilon(\|x-a\|+\|y-b\|)
\end{equation}
whenever $v\in F(a)\cap B(b, \delta)$, $(u^*, v^*)\in N_c({\rm gph}(F), (a,
v))\cap (B_{X^*}\times B_{Y^*})$, and $(x, y)\in {\rm gph}(F)$ with
$\|x-a\|+\|y-b\|<\delta$;

(ii) $\mathcal{L}$-subsmooth at $(a, b)\in {\rm gph}(F)$
if $F^{-1}$ is L-subsmooth at $(b, a)$.

Readers are invited to consult reference \cite{9} for more properties and examples with respect to the concept of L-subsmooth.

\setcounter{equation}{0}

\section{BCQ and strong BCQ for nonconvex (GEC)}

This section is devoted to constraint qualifications of BCQ and strong BCQ for nonconvex (GEC) as well as their equivalences. We first recall the concepts of BCQ and strong BCQ for convex (GEC).

Suppose that $F:X\rightrightarrows Y$ is a convex closed multifunction and $A$ is a convex closed subset of $X$. Recall from \cite{11} that convex (GEC) is said to have the BCQ at $a\in S$, if
\begin{equation}\label{3.1}
  N(S, a)=D^*F(a, b)(Y^*)+N(A, a)
\end{equation}
and
convex (GEC) is said to have the strong BCQ at $a\in S$, if there exists $\tau\in(0, +\infty)$ such that
\begin{equation}\label{3.2}
  N(S, a)\cap B_{X^*}\subset\tau(D^*F(a, b)(B_{Y^*})+N(A, a)\cap B_{X^*}),
\end{equation}
where $S:=\{x\in A: b\in F(x)\}$ is the solution
set of (GEC).

Taking applications of BCQ and strong BCQ into account, we are inspired by \eqref{3.1} and \eqref{3.2} to consider the following forms of BCQ and strong BCQs for nonconvex (GEC).

{\it Let $a\in S$. We say that

{\rm(i)} (GEC) has the BCQ at $a$ in the sense of Clarke, if
\begin{equation}\label{3.3}
  N_c(S, a)\subset D_c^*F(a, b)(Y^*)+N_c(A, a);
\end{equation}

{\rm(ii)} (GEC) has the strong BCQ at $a$ in the sense of Clarke, if there exists $\tau\in(0, +\infty)$ such that
\begin{equation}\label{3.4}
  N_c(S, a)\cap B_{X^*}\subset\tau(D_c^*F(a, b)(B_{Y^*})+N_c(A, a)\cap B_{X^*});
\end{equation}

{\rm(iii)} (GEC) has the strong BCQ at $a$ in the sense of Fr\'echet, if there exists $\tau\in(0, +\infty)$ such that
\begin{equation}\label{3.4a}
  \hat N(S, a)\cap B_{X^*}\subset\tau(D_c^*F(a, b)(B_{Y^*})+N_c(A, a)\cap B_{X^*});
\end{equation}

{\rm(iv)} (GEC) has the strong BCQ at $a$ in the sense of Mordukhovich, if there exists $\tau\in(0, +\infty)$ such that}
\begin{equation}\label{3.4b}
  N(S, a)\cap B_{X^*}\subset\tau(D_c^*F(a, b)(B_{Y^*})+N_c(A, a)\cap B_{X^*}).
\end{equation}

\noindent{\bf Remark 3.1.} It is clear that \eqref{3.4}$\Rightarrow$\eqref{3.4b}$\Rightarrow$\eqref{3.4a}. Furthermore, for the case that $F$ is a convex closed multifunction and $A$ is a convex closed subset, the solution set $S$ is convex, and coderivatives and normal cones are in the sense of convex analysis. Thus, strong BCQs of \eqref{3.4}, \eqref{3.4a} and \eqref{3.4b} reduce to \eqref{3.2}, and BCQ of \eqref{3.3} is equivalent to \eqref{3.1} as the inverse inclusion of \eqref{3.1} holds trivially in this case.

Recall that for the convex inequality defined by a proper lower semicontinuous convex function, Hu \cite{10} introduced one concept of end set to study BCQ and strong BCQ, and used this concept to characterize strong BCQ. Motivated by this, we are interesting in characterizing strong BCQ of \eqref{3.4} for (GEC) in this way. We recall the concept of end set.

Let $C$ be a subset of $X$. Recall from \cite{10} that the end set of
$C$ is defined by
\begin{equation}\label{3.5}
E[C]:=\{z\in \overline{[0, 1]C}: tz\notin \overline{[0,
1]C}\,\,\,\,\forall t>1\}.
\end{equation}
It is shown in \cite{10} that if $C$ is closed and convex then
$$
E[C]=\{z\in C: tz\notin C\,\,\,\,\forall t>1\}.
$$

The following theorem provides an equivalent condition of strong BCQ of \eqref{3.4} for (GEC) by using BCQ and end set.

\begin{them}
Let $z\in S$ and $\tau\in (0, +\infty)$. Then (GEC) has the strong BCQ of \eqref{3.4} at $z$ with constant $\tau>0$ if
and only if (GEC) has the BCQ of \eqref{3.3} at $z$ and
\begin{equation}\label{3.6}
d(0, E[D_c^*F(z, b)(B_{Y^*})+N_c(A, z)\cap B_{X^*}])\geq
\frac{1}{\tau}.
\end{equation}
\end{them}

\begin{proof}
The necessity part. Since BCQ follows from strong BCQ trivially, it suffices to prove \eqref{3.6}. Let $x^*\in E[D_c^*F(z,
b)(B_{Y^*})+N_c(A, z)\cap B_{X^*}]$. Then $x^*\not= 0$. Noting that
$D_c^*F(z, b)(B_{Y^*})+N_c(A, z)\cap B_{X^*}$ is weak$^*$-closed and
convex, it follows that
$$
x^*\in D_c^*F(z, b)(B_{Y^*})+N_c(A, z)\cap B_{X^*}\subset N_c(S, z).
$$
Using the strong BCQ of \eqref{3.4}, one has
$$
\frac{x^*}{\tau\|x^*\|}\in D_c^*F(z, b)(B_{Y^*})+N_c(A, z)\cap
B_{X^*}.
$$
This implies that $\frac{1}{\tau\|x^*\|}\leq 1$ by \eqref{3.5}; that is $\|x^*\|\geq \frac{1}{\tau}$. Thus \eqref{3.6} holds.

 The sufficiency part. Let $x^*\in N_c(S, z)\cap B_{X^*}$. We
set
$$
\lambda:=\sup\{t>0: tx^*\in D_c^*F(z, b)(B_{Y^*})+N_c(A, z)\cap
B_{X^*}\}.
$$
Then, $\lambda>0$ as (GEC) has the BCQ at $z$. If $\lambda=+\infty$, then there exists
$t>\frac{1}{\tau}$ such that
$$
tx^*\in D_c^*F(z, b)(B_{Y^*})+N_c(A, z)\cap B_{X^*}.
$$
This and $0\in D_c^*F(z, b)(B_{Y^*})+N_c(A, z)\cap B_{X^*}$ imply that
$$
x^*\in \tau(D_c^*F(z, b)(B_{Y^*})+N_c(A, z)\cap B_{X^*}).
$$
If $\lambda<+\infty$, then one has
$$
\lambda x^*\in E[D_c^*F(z, b)(B_{Y^*})+N_c(A, z)\cap B_{X^*}].
$$
Thus
$$
\lambda\geq\lambda\|x^*\|\geq d(0, E[D_c^*F(z, b)(B_{Y^*})+N_c(A, z)\cap
B_{X^*}])\geq \frac{1}{\tau},
$$
and consequently
$$
x^*\in\tau(D_c^*F(z, b)(B_{Y^*})+N_c(A, z)\cap B_{X^*}).
$$
The proof is complete.
\end{proof}

Next, we use polar and techniques in dual theory to study BCQ and strong BCQ for (GEC).
Let $M$ and $N$ be subsets of $X$ and $X^*$, respectively. Recall
from \cite{24} that the polar of $M$ and $N$ are defined by
$$
M^{\circ}:=\{x^*\in X^*: \langle x^*, x\rangle\leq 1\;\;\forall x\in
M\}\;\;{\rm{and}}\;\;N^{\circ}:=\{x\in X: \langle x^*, x\rangle\leq
1\;\;\forall x^*\in N\}.
$$
It is known that when $M$ is a cone of $X$, the polar of $M$ reduces
to $M^{\ominus}$ which is called the negative polar of $M$, where
$M^{\ominus}:=\{x^*\in X^*: \langle x^*, x\rangle\leq 0\;\;\forall
x\in M\} $. Hence, Clarke tangent cone and Clarke normal cone are dual; that is
\begin{equation}\label{3.7}
T_c(M, x)^{\circ}=N_c(M, x)\;\;{\rm{and}}\;\;N_c(M,
x)^{\circ}=T_c(M, x)\;\;\forall\, x\in M.
\end{equation}

Using this known relationships \eqref{3.7} and the polar, we establish several results on sufficient/necessary conditions for (GEC) to have BCQ and strong BCQs. These conditions are given by Clarke tangent cone and Clarke tangent derivative.

\begin{them}
Let  $z\in S$ and $\tau\in (0, +\infty)$.

 {\rm (i)} Suppose that
\begin{equation}\label{3.8}
  \overline{D_cF^{-1}(b, z)(\frac{1}{\tau} B_Y)}\cap \overline{T_c(A, z)+\frac{1}{\tau} B_X}\subset\overline{T_c(S, z)+B_X}.
\end{equation}
Then (GEC) has the strong BCQ of \eqref{3.4} at $z$ with constant $\tau>0$.

{\rm (ii)} Suppose that
\begin{equation}\label{3.8a}
  \overline{D_cF^{-1}(b, z)(\frac{1}{\tau} B_Y)}\cap \overline{T_c(A, z)+\frac{1}{\tau} B_X}\subset\overline{T(S, z)+B_X}.
\end{equation}
Then (GEC) has the strong BCQ of \eqref{3.4a} at $z$ with constant $\tau>0$.

{\rm (iii)} Suppose that (GEC) has the strong BCQ of \eqref{3.4} at $z$ with constant $\tau>0$. Then
\begin{equation}\label{3.9a}
  \overline{D_cF^{-1}(b, z)(\eta_1 B_Y)}\cap \overline{T_c(A, z)+\eta_2 B_X}\subset\overline{T_c(S, z)+B_X}
\end{equation}
holds for any $\eta_1, \eta_2\in [0, +\infty)$ with $\eta_1+\eta_2\leq\frac{1}{\tau}$.

{\rm (iv)} Suppose that $X$ is of finite dimension and (GEC) has the strong BCQ of \eqref{3.4a} at $z$ with constant $\tau>0$. Then
\begin{equation}\label{3.9b}
  \overline{D_cF^{-1}(b, z)(\eta_1 B_Y)}\cap \overline{T_c(A, z)+\eta_2 B_X}\subset\overline{{\rm co}(T(S, z))+B_X}
\end{equation}
holds for any $\eta_1, \eta_2\in [0, +\infty)$ with $\eta_1+\eta_2\leq\frac{1}{\tau}$.
\end{them}

\begin{proof}
(i) We first prove that
\begin{equation}\label{3.9}
  \frac{1}{\tau}\big((D_c^*F(z,
b)(B_{Y^*}))^{\circ}\cap(N_c(A, z)\cap B_{X^*})^{\circ}\big)\subset \overline{T_c(S, z)+B_X}.
\end{equation}

Let $v\in \frac{1}{\tau}\big((D_c^*F(z,
b)(B_{Y^*}))^{\circ}\cap(N_c(A, z)\cap B_{X^*})^{\circ}\big)$. By virtue of \cite[Chapter IV, Theorem 1.5]{24}, one has
\begin{equation}\label{3.10}
\tau v\in (N_c(A,
z)\cap B_{X^*})^{\circ}= \overline{\mathrm{co}}\big(T_c(A, z)\cup
B_X\big)=\overline{T_c(A, z)+B_X}.
\end{equation}
 We claim that
\begin{equation}\label{3.11}
(0, \tau v)\in \big(N_c({\rm gph}(F^{-1}), (b, z))\cap (B_{Y^*}\times
X^*)\big)^{\circ}.
\end{equation}
Indeed, for any $(y^*, x^*)$ in $N_c({\rm gph}(F^{-1}), (b, z))\cap (B_{Y^*}\times X^*)$, one has
$$
x^*\in D_c^*F(z, b)(-y^*)\subset D_c^*F(z, b)(B_{Y^*}).
$$
Noting that $\tau v\in (D_c^*F(z, b)(B_{Y^*}))^{\circ}$, it follows that
$$
\langle (y^*, x^*), (0, \tau v) \rangle=\langle x^*, \tau
v\rangle\leq 1.
$$
This implies that \eqref{3.11} holds.

Since
\begin{eqnarray*}
\big(N_c({\rm gph}(F^{-1}), (z, b))\cap (B_{Y^*}\times
X^*)\big)^{\circ}&=&\overline{\mathrm{co}}\big(T_c({\rm gph}(F^{-1}), (b,
z))\cup(B_Y\times\{0\})\big)\\
&=&\overline{T_c({\rm gph}(F^{-1}), (b,
z))+B_Y\times\{0\}},
\end{eqnarray*}
by using \eqref{3.11}, one can prove that $\tau v\in\overline{D_cF^{-1}(b, z)(B_Y)}$ and so $ v\in
\overline{T_c(S, z)+B_X}$ by \eqref{3.10} and \eqref{3.8}. Thus \eqref{3.9} holds.

Noting that
\begin{equation}\label{3.16a}
(D_c^*F(z, b)(B_{Y^*})+N_c(A, z)\cap B_{X^*})^{\circ}\\
\subset D_c^*F(z, b)(B_{Y^*})^{\circ}\cap(N_c(A, z)\cap
B_{X^*})^{\circ}
\end{equation}
(thanks to $0\in D_c^*F(z, b)(B_{Y^*})\cap N_c(A, z)\cap B_{X^*}$),
it follows from \eqref{3.9}, \eqref{3.11} and \cite[Chapter IV, Theorem 1.5]{24}
that
$$
N_c(S, z)\cap B_{X^*}\subset\tau(D_c^*F(z, b)(B_{Y^*})+ N_c(A,
z)\cap B_{X^*}).
$$
This shows that (GEC) has the strong BCQ of \eqref{3.4} with constant $\tau>0$.

(ii) Using the proof of (i), one has
\begin{equation}\label{3.16b}
  \frac{1}{\tau}\big((D_c^*F(z,
b)(B_{Y^*}))^{\circ}\cap(N_c(A, z)\cap B_{X^*})^{\circ}\big)\subset \overline{T(S, z)+B_X}.
\end{equation}
Since $\hat N(S, z)\cap B_{X^*}\subset(T(S, z)+B_X)^{\circ}$, it follows from \eqref{3.16a} and \eqref{3.16b} that
$$
\hat N(S, z)\cap B_{X^*}\subset\tau(D_c^*F(z, b)(B_{Y^*})+ N_c(A,
z)\cap B_{X^*}).
$$
This means that (GEC) has the strong BCQ of \eqref{3.4a} with constant $\tau>0$.

(iii) Let $\eta_1, \eta_2\in [0, +\infty)$ be such that $\eta_1+\eta_2\leq\frac{1}{\tau}$. Suppose to the contrary that there exists one vector $v\in X$ such that
$$
v\in \overline{D_cF^{-1}(b, z)(\eta_1 B_Y)}\cap \overline{T_c(A, z)+\eta_2 B_X}\ \ {\rm but} \ \ v\not\in \overline{T_c(S, z)+B_X}.
$$
By the seperation theorem, there exists $v^*\in X^*$ with $\|v^*\|=1$ such that
\begin{equation}\label{3.12}
  \langle v^*, v\rangle>\sup\{\langle v^*, u\rangle: u\in T_c(S, z)+B_X\}=1.
\end{equation}
This means that $v^*\in N_c(S, z)\cap B_{X^*}$. By virtue of the strong BCQ, there exist $y^*\in B_{Y^*}$, $z^*\in D_c^*F(z, b)(y^*)$ and $x^*\in N_c(A, z)\cap B_{X^*}$ such that
\begin{equation}\label{3.13}
  v^*=\tau(z^*+x^*).
\end{equation}
Note that $v\in \overline{T_c(A, z)+\eta_2 B_X}$ and one can verify that
\begin{equation}\label{3.14}
  \langle x^*, v\rangle\leq \eta_2.
\end{equation}
(thanks to $x^*\in N_c(A, z)\cap B_{X^*}$). Since $v\in \overline{D_cF^{-1}(b, z)(\eta_1 B_Y)}$, there exist $y_n\in B_Y$ and $v_n\in D_cF^{-1}(b, z)(\eta_1y_n)$ such that $v_n\rightarrow v$. Then
$$
\langle (z^*, -y^*), (v_n, \eta_1y_n)\rangle\leq 0.
$$
This implies that $\langle z^*, v_n\rangle\leq \langle y^*, \eta_1y_n\rangle\leq \eta_1$ as $(y^*, y_n)\in  B_{Y^*}\times B_Y $, and consequently $\langle z^*, v\rangle\leq \eta_1$. Using \eqref{3.13} and \eqref{3.14}, one has $\langle v^*, v\rangle\leq 1$, which contradicts \eqref{3.12}.

(iv) Since $X$ is of finite dimension, it follows that $\hat N(S, z)=(T(S, z))^{\circ}$. Using the proof of (iii), one can verify that \eqref{3.9b} holds for any $\eta_1, \eta_2\in [0, +\infty)$ with $\eta_1+\eta_2\leq\frac{1}{\tau}$. The proof is complete.
\end{proof}

\begin{pro}
Let $z\in S$. Then the following inclusions are equivalent:

{\rm (i)} $N_c(S, z)\subset\overline{D_c^*F(z, b)(Y^*)+N_c(A, z)}^{w^*}$;

{\rm (ii)} $T_c(S, z)\supset D_cF^{-1}(b, z)(0)\cap T_c(A, z)$.
\end{pro}

\begin{proof}
We first prove that
\begin{equation}\label{3.15}
  D_cF^{-1}(b, z)(0)=\big(D_c^*F(z, b)(Y^*)\big)^{\circ}.
\end{equation}

Indeed,  the inclusion of \eqref{3.15} holds trivially. Next, we prove the inverse inclusion of \eqref{3.15}. Suppose to the contrary that there exists $h\in\big(D_c^*F(z, b)(Y^*)\big)^{\circ}$ such that $h\not\in D_cF^{-1}(b, z)(0)$; that is, $(h, 0)\not\in T_c({\rm gph}(F), (z, b))$. Applying the seperation theorem, there exists $(h^*, y^*)\in X^*\times Y^*$ with $\|(h^*, y^*)\|=1$ such that
\begin{equation}\label{3.16}
  \langle h^*, h\rangle>\sup\{\langle (h^*, y^*), (u, v)\rangle : (u, v)\in T_c({\rm gph}(F), (z, b))\}=0
\end{equation}
This means that $(h^*, y^*)\in N_c({\rm gph}(F), (z, b))$ and consequently $h^*\in D_c^*F(z, b)(Y^*)$. Hence $\langle h^*, h\rangle\leq 0$ as $D_c^*F(z, b)(Y^*)$ is a cone, which contradicts \eqref{3.16}.

Using \eqref{3.15}, \cite[Chapter IV, Theorem 1.5]{24} and \cite[Lemma 4.1]{BFL}, one can verify that
\begin{eqnarray*}
\big(D_cF^{-1}(b, z)(0)\cap T_c(A, z)\big)^{\circ}&=&\overline{D_cF^{-1}(b, z)(0)^{\circ}+T_c(A, z)^{\circ}}^{w^*}\\
&=&\overline{\overline{D_c^*F(z, b)(Y^*)}^{w^*}+N_c(A, z)}^{w^*}\\
&=&\overline{D_c^*F(z, b)(Y^*)+N_c(A, z)}^{w^*}
\end{eqnarray*}
and
\begin{eqnarray*}
\big(\overline{D_c^*F(z, b)(Y^*)+N_c(A, z)}^{w^*}\big)^{\circ}&=&\big(D_c^*F(z, b)(Y^*)+N_c(A, z)\big)^{\circ}\\
&=&D_cF^{-1}(b, z)(0)\cap T_c(A, z).
\end{eqnarray*}
This means that the equivalence of (i) and (ii) follows. The proof is complete.
\end{proof}

We close this section with the following corollary which is one necessary condition of BCQ for (GEC) and immediate from Proposition 3.1.

\begin{coro}
Let $z\in S$. If (GEC) has the BCQ at $z$, then
$$
D_cF^{-1}(b, z)(0)\cap T_c(A, z)\subset T_c(S, z).
$$
\end{coro}

\setcounter{equation}{0}

\section{Applications to Metric Subregularity of (GEC)}
In this section, we mainly apply BCQ and strong BCQs studied in section 3 to metric subregularity of nonconvex (GEC) and aim to establish necessary and/or sufficient conditions for metric subregularity in terms of these constraint qualifications. We begin with the concept of metric subregularity.

Recall from \cite{9,7} that (GEC) is said to be metrically subregular at $a\in S$ if there exists $\tau\in (0 +\infty)$ such that
\begin{equation}\label{4.1a}
d(x, S)\leq\tau (d(b, F(x))+ d(x,
A))\ \ {\rm for\ all}\  x\ {\rm close\ to \ }a.
\end{equation}

First we establish the following proposition on metric subregularity of (GEC). This result was also obtained by the authors \cite{HYZ}. For the sake of completeness, we give its another different proof which is inspired by \cite[Theorem 3.1]{7}.

\begin{pro}
Let $a\in S$. Suppose that (GEC) is
metrically subregular at $a$. Then there exist $\tau, \delta\in (0, +\infty)$ such
that (GEC) has the strong BCQ of \eqref{3.4a} at all points in $S\cap B(a, \delta)$ with the same constant $\tau>0$.
\end{pro}
\begin{proof}
Since (GEC) is metrically subregular at $a$, there exist
$\tau, r\in (0, +\infty)$ such that
\begin{equation}\label{4.1}
d(x, S)\leq\tau (d(b, F(x))+ d(x,
A))\ \ \forall x\in B(a, r).
\end{equation}
For any $(x, y)\in X\times
Y$, let $\|(x, y)\|_{\tau}:=\frac{\tau+1}{\tau}\|x\|+\|y\|$. Clearly $\|\cdot\|_{\tau}$ is a norm on $X\times
Y$ inducing the product topology, and furthermore the unit ball of
dual space of $(X\times Y, \|\cdot\|_{\tau})$ is
$(\frac{\tau}{\tau+1}B_{X^*})\times B_{Y^*}$. Let $\delta:=\frac{r}{2}$. We claim that
\begin{equation}\label{4.3}
d(x, S)\leq \tau(d_{\|\cdot\|_{\tau}}((x, y), {\rm gph}(F))+\|y-b\|+d(x,
A))\,\,\,\,\forall (x, y)\in B(a, \delta)\times Y.
\end{equation}
Indeed, suppose to the contrary that there exists $(x_0, y_0)\in B(a, \delta)\times Y$ such
that
$$
d(x_0, S)> \tau(d_{\|\cdot\|_{\tau}}((x_0, y_0),
{\rm gph}(F))+\|y_0-b\|+d(x_0, A))
$$
This implies that there exists $u\in X$ such that
$$
d(x_0, S)> \tau(\frac{\tau+1}{\tau}\|u-x_0\|+d(y_0, F(u))+\|y_0-b\|+d(x_0, A)).
$$
Thus,
$$
d(x_0, S)>\|u-x_0\|+\tau(d(b, F(u))+d(u, A)).
$$
Since
$$
\|u-a\|\leq\|u-x_0\|+\|x_0-a\|<d(x_0, S)+\|x_0-a\|\leq
2\|x_0-a\|<r,
$$
using \eqref{4.1}, one has
$$
d(x_0, S)>\|u-x_0\|+d(u, S)\geq d(x_0, S),
$$
which is a contradiction. Hence \eqref{4.3} holds.

 Let $z\in B(a, \delta)\cap S$ and $x^*\in \hat N(S,
z)\cap B_{X^*}$. Then $x^*\in\hat{\partial}d(\cdot, S)(z)$ by \cite[Corollary 1.96]{17} and thus for any
$\varepsilon>0$ there exits $\delta_1\in (0, \delta)$ such that
$$
\langle x^*, x-z\rangle\leq d(x,
S)+\tau\varepsilon\|x-z\|\,\,\,\,\forall x\in B(z, \delta_1).
$$
Noting that $B(z, \delta_1)\subset B(a, r)$, it follows from \eqref{4.3}
that for any $x\in B(z, \delta_1)$,
$$
\langle x^*, x-z\rangle\leq\tau(d_{\|\cdot\|_{\tau}}((x, y),
{\rm gph}(F))+\|y-b\|+d(x, A))+\tau\varepsilon\|x-z\|.
$$
This means
$$(\frac{x^*}{\tau}, 0)\in \hat{\partial}\phi(z,
b)\subset \partial_c\phi(z, b),
$$
where $\phi(x, y):=d_{\|\cdot\|_{\tau}}((x, y),
{\rm gph}(F))+\|y-b\|+d(x, A)$. On the other hand, applying \cite[Proposition 2.4.2 and Theorem 2.9.8]{12}, one has
\begin{eqnarray*}
\partial_c\phi(z, b)&\subset&\partial_cd_{\|\cdot\|_{\tau}}((\cdot, \cdot),
{\rm gph}(F))(z, b)+\{0\}\times B_{Y^*}+N_c(A, z)\cap B_{X^*}\times\{0\}\\
&\subset&N_c({\rm gph}(F), (z, b))+\{0\}\times B_{Y^*}+N_c(A, z)\cap B_{X^*}\times\{0\}.
\end{eqnarray*}
This means that $x^*\in \tau(D_c^*F(z,
b)(B_{Y^*})+N_c(A, z)\cap B_{X^*})$. The proof is complete.
\end{proof}

The following theorems show that the necessary condition for metric subregularity in Proposition 4.1 can be strengthened in finite-dimensional space and Asplund space setting.

\begin{them}
Let $X$ be of finite dimension and $a\in
S$. Suppose that (GEC) is metrically subregular at $a$. Then there
exist $\tau, \delta\in (0, +\infty)$ such that (GEC) has the strong BCQ of \eqref{3.4b} at all points in $S\cap B(a, \delta)$ with the same constant $\tau>0$.

\end{them}

\begin{proof}
Choose $\tau, r\in (0, +\infty)$ such that \eqref{4.1} holds. As the proof in
Proposition 4.1, by defining $\|(x,
y)\|_{\tau}:=\frac{\tau+1}{\tau}\|x\|+\|y\|$ for any $(x, y)\in
X\times Y$, we have that \eqref{4.3} holds wiht $\tau>0$ and $\delta:=\frac{r}{2}>0$. Take any $z\in B(a,
\delta)\cap S$ and $z^*\in N(S, z)\cap B_{X^*}$. Since $X$
is of finite dimension, there exist $z_n\stackrel S\rightarrow z$ and
$\hat z_n^*\rightarrow z^*$ such that $\hat z_n^*\in \hat N(S,
z_n)$. Set $z_n^*:=\frac{\hat z_n^*}{\|\hat z_n^*\|}$. Then $z_n^*\in \hat N(S, z_n)\cap B_{X^*}$ and $z_n^*\rightarrow \frac{z^*}{\|z^*\|}$. Let $\phi(x, y):=d_{\|\cdot\|_{\tau}}((x, y),
{\rm gph}(F))+\|y-b\|+d(x, A)$. Then, for any $n\in
\mathbb{N}$ sufficiently large, there exists $r_n\in (0, \delta-\|z_n-a\|)$
such that
\begin{equation}\label{4.7}
\langle z_n^*, x-z_n\rangle\leq d(x, S)+\tau\varepsilon\|x-z_n\|.
\end{equation}
holds for any $x\in B(z_n, r_n)$. Using \eqref{4.3} and \eqref{4.7}, one has
$$
\langle z_n^*, x-z_n\rangle\leq\tau(\phi(x, y)-\phi(z_n,
b))+\tau\varepsilon\|(x-z_n, y-b)\|_{\tau}\ \ \forall (x, y)\in B(z_n,
r_n)\times B(b, r_n),
$$
and consequently
$$(\frac{z_n^*}{\tau}, 0)\in
\hat{\partial}\phi(z_n, b)\subset\partial_c\phi(z_n, b).
$$
This and \cite[Theorem 2.9.8]{12} imply that there exist $a_n^*\in \partial_cd(\cdot,
A)(z_n)$, $b_n^*\in B_{Y^*}$ and $(x_n^*, y_n^*)\in \partial_cd_{\|\cdot\|_{\tau}}((\cdot,\cdot), {\rm gph}(F))(z_n, b)$ such that
\begin{equation}\label{4.8}
(\frac{z_n^*}{\tau}, 0)=(x_n^*, y_n^*)+(a_n^*, 0)+(0, b_n^*).
\end{equation}
Since $\{(x_n^*, y_n^*)\}$ and
$\{a_n^*\}$ are bounded, without loss of generalization (considering
subnet if necessary), we can assume that
$$
(x_n^*, y_n^*)\stackrel{w^*}\rightarrow(x^*, y^*),
a_n^*\rightarrow
a^*\,\,\,\mathrm{and}\,\,\,b_n^*\stackrel{w^*}\rightarrow b^*\in
B_{Y^*}.
$$
By virtue of \cite[Proposition 2.1.5 and Theorem 2.9.8]{12}, one has
$$
(x^*, y^*)\in
\partial_cd_{\|\cdot\|_{\tau}}((\cdot,\cdot), {\rm gph}(F))(z, b)\subset N_c({\rm gph}(F), (z, b))\  \mathrm{and}\ a^*\in \partial_cd(\cdot,
A)(z).
$$
Taking limits in \eqref{4.8}
with respect to the weak$^*$-topology, one has
$$
(\frac{z^*}{\tau\|z^*\|}, 0)=(x^*, y^*)+(a^*, 0)+(0, b^*).
$$
This implies that $z^*\in \tau(D_c^*F(z, b)(B_{Y^*})+N_c(A, z)\cap
B_{X^*})$ as $\|z^*\|\leq 1$. The proof is complete.
\end{proof}

\begin{them}
Let $X$ be of finite dimension, $Y$ be an Asplund space, and let $a\in S$. Suppose that (GEC) is
metrically subregular at $a$. Then there exist $\tau, \delta\in (0, +\infty)$ such
that
\begin{equation}\label{4.9}
N(S, z)\cap B_{X^*}\subset\tau(D^*F(z, b)(B_{Y^*})+N(A, z)\cap
B_{X^*})\,\,\,\,\forall z\in B(a, \delta)\cap S.
\end{equation}
\end{them}
\begin{proof}
Choose $\tau, r\in (0, +\infty)$ such that \eqref{4.1} holds. Set $\delta:=\frac{r}{2}$, and let $z\in B(a,
\delta)\cap S$ and $z^*\in N(S, z)\cap B_{X^*}$. Since $X$
is of finite dimension, there exist $z_n\stackrel S\rightarrow z$ and
$\hat z_n^*\rightarrow z^*$ such that $\hat z_n^*\in \hat N(S,
z_n)$. Set $z_n^*:=\frac{\hat z_n^*}{\|\hat z_n^*\|}$. Then $z_n^*\in \hat N(S, z_n)\cap B_{X^*}$ and $z_n^*\rightarrow \frac{z^*}{\|z^*\|}$. Let $f(x, y):=\|y-b\|+d(x, A)$ and $\phi(x,
y):=d_{\|\cdot\|_{\tau}}((x, y), {\rm gph}(F))$. As the proof in Theorem
4.1, one has $ (\frac{z_n^*}{\tau}, 0)\in
\hat{\partial}(\phi+f)(z_n, b)$. Noting that $X$ is of finite dimension and $Y$ is Asplund
space, it follows from
\cite[Theorem 2.33]{17} that there exist $(x_n, y_n), (u_n, v_n)\in
B(z_n, \frac{1}{n})\times B(b, \frac{1}{n})$ with $|\phi(x_n,
y_n)-\phi(z_n, b)|<\frac{1}{n}$ and $|f(u_n, v_n)-f(z_n,
b)|<\frac{1}{n}$ such that
\begin{equation}\label{4.10}
(\frac{z_n^*}{\tau}, 0)\in \hat{\partial}\phi(x_n,
y_n)+\hat{\partial}f(u_n, v_n)+\frac{1}{n}(B_{X^*}\times B_{Y^*}).
\end{equation}
Applying \cite[Theorem 2.33]{17} again to $\hat{\partial}f(u_n, v_n)$,
there exist $(a_n, b_n)\in B(u_n, \frac{1}{n})\times B(v_n,
\frac{1}{n})$ such that
\begin{equation}\label{4.11}
\hat{\partial}f(u_n, v_n)\subset\hat{\partial}d(\cdot,
A)(a_n)\times\{0\}+\{0\}\times\partial\|\cdot-b\|(b_n)+\frac{1}{n}(B_{X^*}\times
B_{Y^*}).
\end{equation}
By \eqref{4.10} and \eqref{4.11}, there exist $(x_n^*,
y_n^*)\in\hat{\partial}\phi(x_n, y_n)$, $a_n^*\in
\hat{\partial}d(\cdot, A)(a_n)$ and $b_n^*\in
B_{Y^*}$ such that
\begin{equation}\label{4.12}
(\frac{z_n^*}{\tau}, 0)\in (x_n^*, y_n^*)+(a_n^*, 0)+(0,
b_n^*)+\frac{2}{n}(B_{X^*}\times B_{Y^*}).
\end{equation}

Let $(\bar x_n^*, \bar y_n^*):=\frac{\tau}{\tau+1}(x_n^*, y_n^*)$. As $(x_n^*, y_n^*)\in\hat{\partial}\phi(x_n, y_n)$,  Lemma 2.2 implies that there exist $(\widetilde x_n,
\widetilde{y}_n)\in {\rm gph}(F)$ and $(\hat x_n^*,
\hat{y}_n^*)\in \hat N_{\|\cdot\|_{\tau}}({\rm gph}(F), (\widetilde x_n, \widetilde y_n))$ such that
\begin{eqnarray}\label{4.13}
\|(\widetilde x_n, \widetilde{y}_n)-(x_n, y_n)\|< \phi(x_n, y_n)+\frac{1}{n}\ \ {\rm and} \ \
\|(\hat x_n^*, \hat{y}_n^*)-(x_n^*,
y_n^*)\|<\frac{1}{n}.
\end{eqnarray}
Note that
\begin{equation}\label{4.11a}
(\widetilde x_n^*, \widetilde{y}_n^*):=\frac{\tau}{\tau+1}(\hat{x}_n^*, \hat{y}_n^*)\in \hat N({\rm gph}(F), (\widetilde x_n, \widetilde y_n))
\end{equation}
thanks to $(\hat x_n^*,
\hat{y}_n^*)\in \hat N_{\|\cdot\|_{\tau}}({\rm gph}(F), (\widetilde x_n, \widetilde y_n))$ and consequently
\begin{equation}\label{4.12a}
  \|(\widetilde x_n^*, \widetilde{y}_n^*)-(\bar x_n^*,
\bar{y}_n^*)\|<\frac{\tau}{n(\tau+1)}.
\end{equation}

Applying Lemma 2.2 to $a_n^*\in \hat{\partial}d(\cdot, A)(a_n)$,
there exist $\widetilde a_n\in A$ and $\widetilde a_n^*\in \hat N(A,
\widetilde a_n)$ such that
\begin{equation}\label{4.14}
\|\widetilde a_n-a_n\|< d((a_n,
A)+\frac{1}{n}\,\,\,\mathrm{and}\,\,\, \|\widetilde
a_n^*-a_n^*\|<\frac{1}{n}
\end{equation}
Since $(\bar x_n^*, \bar y_n^*)\in \frac{\tau}{\tau+1}\hat{\partial}d_{\|\cdot\|_{\tau}}((\cdot,\cdot),
{\rm gph}(F))(x_n, y_n)\subset B_{X^*}\times B_{Y^*}$, $a_n^*\in
\hat{\partial}d(\cdot, A)(a_n)\subset B_{X^*}$, $b_n^*\in B_{Y^*}$
and $B_{X^*}\times B_{Y^*}$ is sequentially weak$^*$-compact (as $X$ is of finite dimension and $Y$ is an Asplund space),  without loss of generalization (consider subsequence if necessary), we can assume that
$$
(\bar x_n^*, \bar y_n^*)\stackrel{w^*}\rightarrow (x^*, y^*),
a_n^*\stackrel{w^*}\rightarrow a^*\in
B_{X^*}\,\,\,\mathrm{and}\,\,\,b_n^*\stackrel{w^*}\rightarrow b^*\in
B_{Y^*}.
$$
Using \eqref{4.13}-\eqref{4.14}, one has
$$(\widetilde x_n^*,
\widetilde y_n^*)\stackrel{w^*}\rightarrow (x^*, y^*),  \
\widetilde a_n^*\stackrel{w^*}\rightarrow a^*, \ (\widetilde x_n, \widetilde y_n)\stackrel{{\rm gph}(F)}\longrightarrow
(z, b) \ \ {\rm and}  \ \ \widetilde a_n\stackrel A\rightarrow z.
$$
This implies that $(x^*,
y^*)\in N({\rm gph}(F), (z, b))$ and $a^*\in N(A, z)\cap B_{X^*}$. Taking limits as $n\rightarrow\infty$ with respect to the
weak$^*$-topology in \eqref{4.12}, one has
\begin{equation}\label{4.15}
 (\frac{z^*}{\tau\|z^*\|}, 0)\in
\frac{\tau+1}{\tau}(x^*, y^*)+(a^*, 0)+(0, b^*).
\end{equation}
Since $b^*\in B_{Y^*}$ and $\frac{\tau+1}{\tau}(x^*, y^*)\in
N({\rm gph}(F), (z, b))$, it
follows from \eqref{4.15} that
$$
z^*\in \tau(DF(z, b)(B_{Y^*})+N(A, z)\cap B_{X^*})
$$
as $\|z^*\|\leq 1$. The proof is complete.
\end{proof}

Using the proof of Theorem 4.2, we have the following corollary.
\begin{coro}
Let $X, Y$ be Asplund spaces and $a\in S$. Suppose that (GEC) is
metrically subregular at $a$. Then there exist $\tau, \delta\in (0, +\infty)$ such
that
\begin{equation}\label{4.16}
  \hat N(S, z)\cap B_{X^*}\subset\tau(D^*F(z, b)(B_{Y^*})+N(A, z)\cap
B_{X^*})\,\,\,\,\forall z\in B(a, \delta)\cap S.
\end{equation}
\end{coro}

\noindent{\bf Remark 4.1.} Theorem 4.3 and Corollary 4.1 are results on necessary conditions for metric subregularity of (GEC) in the Asplund and finite-dimensional spaces setting, and these necessary condition forms in \eqref{4.9} and \eqref{4.16} are similar to the strong BCQ of \eqref{3.4}, \eqref{3.4a} and \eqref{3.4b} for (GEC). It is an idea to consider and study these types of constraint qualifications for (GEC) as well as equivalent conditions for them.

Next, we focus on sufficient conditions for metric
subregularity given by constraint qualifications. It is known from the counterexample of \cite[Example 4.5]{9} that the general nonconvex (GEC) may not have metric subregularity even with the assumption of strong BCQ of \eqref{3.4} in the finite-dimensional space. Thus, we consider the (GEC) defined by an L-subsmooth multifunction and a submsooth subset. The following theorem is inspired by \cite[Theorem 4.4]{9} and similar to \cite[Theorem 4.3]{HYZ}. We give its proof for the sake of completeness.

\begin{them}
Let $a\in S$. Suppose that $F$ is
$\mathcal{L}$-subsmooth at
$(a, b)$, $A$ is subsmooth at $a$ and that there exist
$\tau, \delta\in (0, +\infty)$ such that (GEC) has the strong BCQ of \eqref{3.4} at all points in $S\cap B(a, \delta)$ with the same constant $\tau>0$. Then (GEC) is metrically subregular at $a$.

\end{them}
\begin{proof} Let $\varepsilon\in (0,
\frac{1}{2\tau+1})$. Since $F$ is $\mathcal{L}$-subsmooth at $(a,
b)$ and $A$ is subsmooth at $a$, there exists $r\in (0, \delta)$
such that
\begin{equation}\label{4.17}
\langle (v^*, u^*), (y-b, x-u)\rangle\leq
\varepsilon(\|y-b\|+\|x-u\|)
\end{equation}
and
\begin{equation}\label{4.18}
 \langle \bar x^*, z-\bar x\rangle\leq
d(z, A)+\varepsilon\|z-\bar x\|\,\,\,\,\forall z\in B(a, r)
\end{equation}
hold for any $u\in F^{-1}(b)\cap B(a, r)$, $(v^*, u^*)\in
N_c({\rm gph}(F^{-1}), (b, u))\cap (B_{Y^*}\times B_{X^*})$, $\bar x\in
A\cap B(a, r)$, $\bar x^*\in N_c(A, \bar x)\cap B_{X^*}$ and $(y,
x)\in {\rm gph}(F^{-1})$ with $\|y-b\| +\|x-a\|<r$.

We prove that there exist $\tau_1, \delta_1>0$ such that
\begin{equation}\label{4.19a}
  d(x, S)\leq\tau_1(d(b, F(x))+d(x, A))\,\,\,\,\forall x\in B(a,
\delta_1).
\end{equation}
Granting this, it follows that (GEC) is metrically subregular at $a$.

Let $\tau_1:=\frac{(\tau+1)\varepsilon+\tau}{1-(2\tau+1)\varepsilon}$ and choose $\delta_1\in (0, \frac{r}{2})$ such that $\delta_1<\tau_1r$. Let $x\in B(a, \delta_1)\backslash S$.
Then $d(x, S)\leq\|x-a\|<\delta_1$. Take any
$\gamma\in(\max\{(2\tau+1)\varepsilon, \frac{d(x,
S)}{\delta_1}\}, 1)$. By Lemma 2.1, there exist $u\in S$
and $u^*\in N_c(S, u)\cap B_{X^*}$ with $\|u^*\|=1$ such that
\begin{equation}\label{4.20}
\gamma\|x-u\|<\min\{d(x, S), \langle u^*, x-u\rangle\}.
\end{equation}
Noting that $\|u-a\|\leq \|u-x\|+\|x-a\|<\frac{d(x,
S)}{\gamma}+\delta_1<r$, it follows from the strong BCQ of \eqref{3.4} that there exist
$y^*\in B_{Y^*}$ and $(x_1^*, x_2^*)\in D_c^*F(u, b)(y^*)\times
(N_c(A, u)\cap B_{X^*})$ such that
\begin{equation}\label{4.21}
u^*=\tau(x_1^*+x_2^*).
\end{equation}
Noting that $(x_1^*, -y^*)\in N_c({\rm gph}(F), (u,b))$, it follows from \eqref{4.17} and \eqref{4.18} that
\begin{equation}\label{4.22}
\langle x_1^*, \widetilde x-u\rangle\ -\langle y^*, \widetilde
y-b\rangle\leq \frac{1+\tau}{\tau}\varepsilon(\|\widetilde
y-b\|+\|\widetilde x-u\|)
\end{equation}
and
\begin{equation}\label{4.23}
\langle x_2^*, x-u\rangle\leq d(x, A)+\varepsilon\|x-u\|
\end{equation}
hold for any
$(\widetilde y, \widetilde x)\in {\rm gph}(F^{-1})\cap B(b, r)\times
B(a, \delta_1)$.

If $F(x)\cap B(b, r)=\emptyset$, then $d(b, F(x))>
r$ and thus
\begin{equation}\label{4.24}
d(x, S)\leq\|x-a\|\leq \delta_1<\tau_1(d(b, F(x))+d(x, A)).
\end{equation}
 Next, we assume that $F(x)\cap B(b, r)\neq\emptyset$. Using \eqref{4.22}, one has
\begin{equation*}\label{4.25}
\langle x_1^*, x-u\rangle\ \leq\| y-b\|+
\frac{1+\tau}{\tau}\varepsilon(\|y-b\|+\|x-u\|)
\end{equation*}
holds for any $y\in F(x)\cap B(b, r)$. This and $d(b, F(x))=d(b, F(x)\cap B(b, r))$ imply that
\begin{equation}\label{4.26}
\langle x_1^*, x-u\rangle\ \leq
(\frac{1+\tau}{\tau}\varepsilon+1)d(b,
F(x))+\frac{1+\tau}{\tau}\varepsilon\|x-u\|.
\end{equation}
By \eqref{4.20}, \eqref{4.21}, \eqref{4.23} and \eqref{4.26}, one has
\begin{equation*}
\gamma\|x-u\|\leq ((\tau+1)\varepsilon+\tau)(d(b, F(x))+d(x,
A))+(2\tau+1)\varepsilon\|x-u\|.
\end{equation*}
This means that
\begin{equation*}
d(x, S)\leq
\frac{(\tau+1)\varepsilon+\tau}{\gamma-(2\tau+1)\varepsilon}(d(b,
F(x))+d(x, A)).
\end{equation*}
Taking limits as $\gamma\rightarrow 1^-$ and using
\eqref{4.24}, one has
\begin{equation*}
d(x, S)\leq \tau_1(d(b, F(x))+d(x,
A)).
\end{equation*}
Hence \eqref{4.19a} holds. The proof is complete.
\end{proof}

The following theorem provides one
characterization for metric subregularity of (GEC) defined by the L-subsmooth multifunction and the submsooth subset in the Asplund space.  The proof can be obtained by using Proposition 4.1, Theorem 4.3 and Lemma 2.1 of the Asplund space version.
\begin{them}
Let $X$ be an Asplund space and $a\in
S$. Suppose that $F$ is $\mathcal{L}$-subsmooth at $(a, b)$ and $A$
is subsmooth at $a$. Then (GEC) is metrically subregular at
$a$ if and only if there exist $\tau, \delta\in (0, +\infty)$ such that
(GEC) has the strong BCQ of \eqref{3.4a} at all points in $S\cap B(a, \delta)$ with the same constant $\tau>0$.
\end{them}

\noindent{\bf Remark 4.2} Let $a\in S$. We define $\tau(F, a, b; A):=\inf\{\tau>0: \eqref{4.1a}\,\,
\mathrm{holds}\}$. For any $z\in S$, we define
\begin{eqnarray*}
\gamma(F, z, b; A):=\inf\{\tau>0: {\rm (GEC) \ has\ the\ strong\ BCQ \ of \ \eqref{3.4a} \ at\ }z\ {\rm with\ }\tau\},  \\
\gamma_c(F,z,b;A):=\inf\{\tau>0: {\rm (GEC) \ has\ the\ strong\ BCQ \ of \ \eqref{3.4} \ at\ }z\ {\rm with\ }\tau\}.\,
\end{eqnarray*}
By the proof of Proposition 4,1, one can verify that
$$
\tau(F, a, b; A)\geq\limsup_{z\stackrel{S}\longrightarrow
a}\gamma(F, z, b; A)\geq\gamma(F, a, b; A).
$$
If $F$ is $\mathcal{L}$-subsmooth at $(a, b)$ and $A$ is
subsmooth at $a$, using the proof of Theorem 4.3, one has
$$
\limsup_{z\stackrel{S}\longrightarrow
a}\gamma_c(F, z, b; A)\geq\tau(F, a, b; A).
$$
In addition, if $X$ is an Asplund space, by Theorem 4.4, one has
$$
\tau(F, a, b; A)=\limsup_{z\stackrel{S}\longrightarrow
a}\gamma(F, z, b; A).
$$

\noindent{\bf Acknowledgment.} This research was supported by the National Natural Science Foundations of P. R. China (Grant No. 11261067 and No. 11371312) and by the IRTSTYN.


\end{document}